\documentclass[a4paper,reqno]{amsart}
\usepackage[utf8]{inputenc}
\usepackage{amssymb}
\usepackage{amsmath}
\usepackage{mathtools}
\usepackage{ifpdf}
\usepackage{vmargin}

\usepackage{graphicx}  
\usepackage{subfigure}
\usepackage{epstopdf}
\usepackage{caption}

\usepackage{setspace}
\usepackage{xcolor}

\ifpdf 
\usepackage[pdftex]{hyperref}
\else
\usepackage[hypertex]{hyperref}
\fi

\hypersetup{
	pdftitle={ },
	pdfauthor={},
}

\newtheorem{theorem}{Theorem}
\newtheorem{proposition}[theorem]{Proposition}

\newtheorem{corollary}[theorem]{Corollary}
\newtheorem{definition}[theorem]{Definition}

\numberwithin{equation}{section}

\newcommand{\R}{\mathbb{R}}
\newcommand{\C}{\mathbb C}
\newcommand{\Z}{\mathbb Z}

\newcommand{\bee}{\begin{eqnarray*}}
	\newcommand{\eee}{\end{eqnarray*}}

\numberwithin{equation}{section}
\numberwithin{theorem}{section}

\begin{document}

 \onehalfspacing

\title[Global well-posedness of the defocusing NLW outside of a ball]{Global well-posedness of the defocusing nonlinear wave equation outside of a ball with radial data for $3<p<5$}
\author[Guixiang Xu]{Guixiang Xu}
\address{\hskip-1.15em Guixiang Xu 
	\hfill\newline Laboratory of Mathematics and Complex Systems,
	\hfill\newline Ministry of Education,
	\hfill\newline School of Mathematical Sciences,
	\hfill\newline Beijing Normal University,
	\hfill\newline Beijing, 100875, People's Republic of China.}
\email{guixiang@bnu.edu.cn}

\author[Pengxuan Yang]{Pengxuan Yang}
\address{\hskip-1.15em Pengxuan Yang 
	\hfill\newline Laboratory of Mathematics and Complex Systems,
	\hfill\newline Ministry of Education,
	\hfill\newline School of Mathematical Sciences,
	\hfill\newline Beijing Normal University,
	\hfill\newline Beijing, 100875, People's Republic of China.}
\email{pxyang@mail.bnu.edu.cn}

\author[Zhuohui You]{Zhuohui You}
\address{\hskip-1.15em Zhuohui You
	\hfill\newline Laboratory of Mathematics and Complex Systems,
	\hfill\newline Ministry of Education,
	\hfill\newline School of Mathematical Sciences,
	\hfill\newline Beijing Normal University,
	\hfill\newline Beijing, 100875, People's Republic of China.}
\email{zhuohui@mail.bnu.edu.cn}

\date{}

\subjclass[2010]{Primary: 35L05; Secondary: 35E05, 35L20}

\keywords{Dispersive estimate; Distorted Fourier transform; Exterior domain; Radial Sobolev inequality; Strichartz estimate; Wave equation} 

\begin{abstract}
We continue the study of the Dirichlet boundary value problem of nonlinear wave equation with radial data in the exterior $\Omega = \mathbb{R}^3\backslash \bar{B}(0,1)$.  We combine the distorted Fourier truncation method in \cite{Bourgain98:FTM}, the global-in-time (endpoint) Strichartz estimates in \cite{XuYang:NLW} with the energy method in \cite{GallPlan03:NLW} to prove the global well-posedness of the radial solution to the defocusing, energy-subcriticial nonlinear wave equation outside of a ball in  $\left(\dot H^{s}_{D}(\Omega) \cap L^{p+1}(\Omega) \right)\times \dot H^{s-1}_{D}(\Omega)$ with  $1-\frac{(p+3)(1-s_c)}{4(2p-3)}<s<1$, $s_c=\frac{3}{2}-\frac{2}{p-1} $, which extends the result for the cubic nonlinearity in \cite{XuYang:NLW} to the case $3<p<5$.  Except from the argument in \cite{XuYang:NLW}, another new ingredient is that we need make use of the radial Sobolev inequality to deal with the super-conformal nonlinearity in addition to the Sobolev inequality. 
\end{abstract}

\maketitle



\section{Introduction}

In this paper, we continue to study the defocusing, energy-subcriticial nonlinear wave equation
\begin{equation}\label{eq:pnlw}
    \begin{cases}
        \partial_t^2u-\Delta u+|u|^{p-1}u=0, &(t,x)\in I\times \Omega\\
        u(0,x)=u_0(x), &x\in \Omega\\
        \partial_t u(0,x)=u_1(x), &x\in \Omega\\
        u(t,x)=0, &x\in \partial\Omega
    \end{cases}
\end{equation}
where $3<p<5$, $0\in I$, the spatial domain $\Omega=\mathbb{R}^3\setminus \overline{B}(0,1)$, the function $u:I\times \Omega \longrightarrow \mathbb{R}$ and initial data $(u_0,u_1)$ is radial and belongs to  $\left(\dot{H}_D^s(\Omega) \cap L^{p+1}(\Omega)\right) \times \dot{H}_D^{s-1}(\Omega)$. The Sobolev space $\dot H^s_D(\Omega)$ of fractional order with Dirichlet boundary value will be introduced in Section \ref{sect:DFT and ST est}.

The smooth solution of \eqref{eq:pnlw} obeys the energy conservation law $E(u, u_t)(t) = E(u, u_t)(0)$ for any $t\in I$, where
\begin{equation*}
        E(u, u_t)(t):=\int_\Omega \frac{1}{2}|u_t(t)|^2+\frac{1}{2}|\nabla u(t)|^2+\frac{1}{p+1}|u(t)|^{p+1}\,dx.
\end{equation*}

There are lots of works focusing on the Cauchy problem of semilinear wave equation including the energy (sub)-critical wave equation (see   \cite{Dod:NLW:sct1, Dod:NLW:sct2, Dod:NLW:sct3, Dod:NLW:sct4,  JendrejL:NLW:SRC, Kenig:book, LindSog95:NLW, ShatS:NLW, Sogge95:book} and reference therein),  and there are also many important results about the energy-critical wave/Schr\"odinger equations outside of the domain (see \cite{Burq03: NLW CPDE, BurqLP:NLW, BurqP:def NLW: Neum prob,  DuyLaf:NLW:Neum prob, DuyYang24:NLW, KVZ16:NLS:Scat, LiSZ:NLS, LiXZ:NLS, SmithSogge95:NLW JAMS, SmithSogge00:NLW CPDE}  and reference therein).  The energy conservation law plays  a crutial rule in long time dynamics in both of the Cauchy problem (IVP) and the initial boundary value problem (IBVP) of the energy-critical wave/Schr\"odinger equations.

 We now give the definition of local-well-posedness of \eqref{eq:pnlw} in  $\dot H^{s}_{D}(\Omega) \times \dot H^{s-1}_{D}(\Omega)$. 
The equation \eqref{eq:pnlw} is said to be locally well-posed in $\dot H^{s}_{D}(\Omega) \times \dot H^{s-1}_{D}(\Omega)$  if there exists an open interval $I\subset \mathbb{R}$ containing $0$ such that:
$(1)$ there exists a unique solution in $\dot H^{s}_{D}(\Omega) \times \dot H^{s-1}_{D}(\Omega)$;
$(2)$ the solution is continuous in time, which means $(u, \partial_t u)\in C\left(I, \dot H^{s}_{D}(\Omega)\times \dot H^{s-1}_{D}(\Omega)\right)$;
$(3)$ the solution depends continuously on initial data $(u_0,u_1)$.

By the Strichartz estimate in \cite{DuyYang24:LW Str} (or Proposition \ref{prop: ST est} in \cite{XuYang:NLW} in the radial case), and standard Picard fixed point argument, we can obtain the following local well-posedness result of \eqref{eq:pnlw} in $\dot H^s_{D}(\Omega) \times \dot H^{s-1}_{D}(\Omega)$, $s_c\leq s< \frac32$.
\begin{theorem}\label{thm:nlu lwp}
	Let $3\leq p<5$ and $s_c=\frac{3}{2}-\frac{2}{p-1}$, the equation \eqref{eq:pnlw} is locally well-posed in  $\dot H^{s_c}_{D}(\Omega)\times \dot H^{s_c-1}_{D}(\Omega)$ on some interval $I=\left(-T, T\right)$. Moreover, the regularity of initial data is enough to give a lower bound on the time of well-posedness, that is, there exists some positive lifespan $T = T\big(\|(u_0, u_1)\|_{\dot H^{s}_{D}(\Omega) \times \dot H^{s-1}_{D}(\Omega)}\big)$ for any $s_c <s<\frac32$.
\end{theorem}

%

In fact, the dispersive estimate and the (local-in-time) Strichartz estimate  for the wave/Schr\"odinger equations on the (exterior) domain themselves are extremely complicated, we can refer to \cite{BlairSS:NLS:St, Burq03: NLW CPDE, BurqLP:NLW, DuyYang24:LW Str, HidanoMSSZ:AST est, Iv:NLS St est, IvLasLebP:LW:dis est, IvLebeauP:LW:dis est, LiuSongZheng:NLS:exterior,  SmithSogge95:NLW JAMS, SmithSogge00:NLW CPDE, SmithSW:AST est, StaT:NLS St est}  and reference therein. 
 
The local well-posedness theory together with the energy conservation law implies that

\begin{theorem}
		Let $3\leq p<5$, the equation \eqref{eq:pnlw} is globally well-posed in  in $\left(\dot H^1_{D}(\Omega) \cap L^{p+1}(\Omega)\right) \times   L^2(\Omega)$.
	
\end{theorem}

%

We combine the distorted Fourier truncation method, the global-in-time (endpoint) Strichartz estimates for radial data in \cite{XuYang:NLW}, the  (radial) Sobolev inequality with the energy method to prove the global well-posedness of the radial solution to the defocusing, cubic nonlinear wave equation outside of a ball in  $\left(\dot H^{s}_{D}(\Omega) \cap L^{p+1}(\Omega) \right)\times \dot H^{s-1}_{D}(\Omega)$ for $\frac{3}{4}<s<1$, where the solution maybe have infinite energy.

\begin{theorem}[\cite{XuYang:NLW}]\label{thm:nlu gwp} Let $p=3$, and $\frac{3}{4}<s<1$, then
	the equation \eqref{eq:pnlw} with radial data is global well-posedness in $\left(\dot H^{s}_{D}(\Omega) \cap L^{p+1}(\Omega) \right) \times \dot H^{s-1}_{D}(\Omega)$. More precisely,  for arbitrarily large time $T$, the solution obeys the following estimate
	\begin{align*}
	\left \| u \right\|_{C\left([0, T); \; \dot H^{s}_{D}(\Omega)\right)} \lesssim T^{\frac{3(1-s)(2s-1)}{4s-3}}.
	\end{align*}
\end{theorem}

In this paper, we extend the above result from the conformal critical case $p=3$ to  the super-conformal crtical case $3<p<5$.   Main result in this paper is the following.

\begin{theorem}\label{thm:GWP}
    Let $3<p<5$, $s_c=\frac{3}{2}-\frac{2}{p-1}$ and $1-\frac{(p+3)(1-s_c)}{4(2p-3)}<s<1$, then
	the equation \eqref{eq:pnlw} with radial data is global well-posedness in $\left(\dot H^{s}_{D}(\Omega) \cap L^{p+1}(\Omega) \right) \times \dot H^{s-1}_{D}(\Omega)$. More preciesly,  for arbitrarily large time $T$, the solution obeys the following estimate
	\begin{align*}
	 \| u \|_{L^\infty_t \left([0,T); \dot H^{s}_{D}( \Omega)\right)}\lesssim T^{2(1-s)+\frac{\left[7p-3-(6p-6)s\right](1-s)}{2\left[4(2p-3)s-4(2p-3)+(p+3)(1-s_c)\right]} }.
	\end{align*}
\end{theorem}

In addition to the distorted Fourier truncation method, the Sobolev inequality and the radial endpoint Strichartz estimate in \cite{XuYang:NLW}, the new ingredient is that we need use the radial Sobolev inequality (see Proposition \ref{prop:rad sob est}) to deal with the super-conformal critical nonlinearity. 

Lastly, we organize the paper as follows. In Section \ref{sect:DFT and ST est}, we recall the distorted Fourier transform in \cite{LiSZ:NLS,  XuYang:NLW} and  the associated function space and the Littlewood-Paley theory associated to the Dirichlet-Laplacian operator $\Delta_{\Omega}$ in \cite{KVZ16:IMRN Riesz tf, XuYang:NLW}, in addtion, we recall the global-in-time (endpoint) Strichartz estimate with radial data outside of a ball in \cite{XuYang:NLW} and establish the radial Sobolev inequality in Proposition \ref{prop:rad sob est}. In Section \ref{sect:pf rest}, we follow the argument in \cite{GallPlan03:NLW, XuYang:NLW} to combine the Fourier truncation method, the energy method, the (endpoint) Strichartz estimate for radial data with the radial Sobolev inequality to prove the result in Theorem \ref{thm:GWP}.  

\noindent \subsection*{Acknowledgements.}

The authors are supported  by National Key Research and Development Program of China (No. 2020YFA0712900) and by NSFC (No. 12371240, No. 12431008).

\section{distorted Fourier transform and Strichartz estimates}\label{sect:DFT and ST est}
In this section, we will give some notation used in this paper, and recall the distorted Fourier transform on $\Omega=\mathbb{R}^3\setminus \overline{B}(0,1)$,  the  global-in-time Strichartz estimate, and the endpoint Strichartz estimate  for the radial case in \cite{XuYang:NLW} and establish radial Sobolev inequality in Proposition \ref{prop:rad sob est}.

\subsection{Notation} We denote the 3D Dirichlet-Laplacian operator out of the ball by $-\Delta_\Omega$. Let $I\subset \mathbb{R}$ be an interval of time, we write $L_t^qL_x^p(I\times \Omega)$ denote the Banach space with norm
\[\|u\|_{L_t^qL_x^p(I\times \Omega)}=\left(\int_I\left(\int_\Omega |u(t,x)|^p\,dx\right)^\frac{q}{p}\right)^\frac{1}{q}.\]

We write $X\lesssim Y$ or $Y\gtrsim X$ to indicate $X\leqslant CY$ for some constant $C>0$ not depending on $X$ or $Y$. We use $\mathcal{O}(Y)$ to denote any quantity $X$ such that $|X|\lesssim Y$. We use the notation $X\approx Y$, if satisfying $X\lesssim Y\lesssim X$. For any positive number $p\in [1,\infty]$,we write $p'=\frac{p}{p-1}$ as the H\"older conjugate exponent of $p$.  

\subsection{Distorted Fourier transform} Now we recall the distorted Fourier transform associated to the Dirichlet-Laplacian operator $-\Delta_\Omega$ in \cite{LiSZ:NLS, XuYang:NLW}.  The spectral resolution for radial functions on $\Omega=\mathbb{R}^3\setminus \overline{B}(0,1)$ is expressed by  radial, generalized eigenfunctions
\[-\Delta_\Omega e_\lambda=\lambda^2 e_\lambda\]
for $\lambda>0$ which satisfy the Sommerfeld radiation condition, namely
\begin{equation}
    e_\lambda(r)=\frac{\sin \lambda(r-1)}{r},\,\,\, r\geqslant 1.
\end{equation}
For the radial, tempered distributions $u\in \mathcal{S}'(\mathbb{R}^3)$ which supported on $\Omega$, we denote the distorted Fourier transform by
\begin{equation}
    \mathcal{F}_Du(\lambda)=\frac{\sqrt{2}}{\sqrt{\pi}}\int_1^\infty \frac{\sin \lambda(s-1)}{s}u(s)s^2\,ds.
\end{equation}

Note that  the following resolution of identity
\begin{equation*}
            \frac{2}{\pi}\int_{-\infty}^\infty e_\lambda(r)e_\lambda(s)\,d\lambda=\frac{\delta(r-s)}{s^2}
\end{equation*}
from which it follows that $$\mathcal{F}_D^{-1}\mathcal{F}_D u = u$$ 
for the radial function $u\in  \mathcal{S}(\mathbb{R}^3)$ supported in $\Omega$,  where $\mathcal{F}_D^{-1}$ is the formal adjoint, defined on the tempered distributions $v$ as the restriction to $r= |x|\geqslant 1$ of
\begin{equation}
    \mathcal{F}_D^{-1}v(r)=\frac{\sqrt{2}}{\sqrt{\pi}}\int_0^\infty \frac{\sin \lambda(r-1)}{r}v(\lambda)\,d\lambda.
\end{equation}
Consequently, $u\rightarrow \mathcal{F}_Du$ induces an isometric map
\[\mathcal{F}_D: L^2([1,\infty),s^2ds)\rightarrow L^2([0,\infty),d\lambda).\]

Given a bounded function $m(\lambda)$, which for convenience we assume to be defined on all of $\mathbb{R}$ and even in $\lambda$, and radial function $u\in C_c^\infty(\Omega)$, we define
\begin{equation}\label{m}
    m(\sqrt{-\Delta_\Omega})u(r)=\mathcal{F}_D^{-1}(m(\cdot)\mathcal{F}_Du)(r).
\end{equation}
This defines a functional calculus on $L^2_{rad}(\Omega)$ and takes the expression as
\begin{align*}
m(\sqrt{-\Delta_{\Omega}}) f (r) = \int^{\infty}_{1} K_m(r, s) f(s)\, s^2 \, ds
\end{align*}
with 
\begin{align*}
K_m(r, s) = \frac{2}{\pi}\, \int^{\infty}_{0} e_{\lambda}(r)\, e_{\lambda}(s)\, m(\lambda)\; d\lambda. 
\end{align*}

In general, we have the following Mikhlin Multiplier theorem.
\begin{theorem}[\cite{KVZ16:IMRN Riesz tf, LiSZ:NLS, XuYang:NLW}]\label{thm:multiplier}
	Suppose $m: [0, \infty) \rightarrow \C$ obeys 
	\begin{align*}
	\left|\partial^{k}_{\lambda} m (\lambda)\right| \lesssim \lambda^{-k}
	\end{align*}
	for all integer $ k \in [0, 2] $. Then $m(\sqrt{-\Delta_{\Omega}}) $, which we define via the $L^2$ functional calculus, extends uniquely from $L^2(\Omega)\cap L^{p}(\Omega)$ to a bounded operator on $L^p(\Omega)$, for all $1<p<\infty$.
\end{theorem}

With the help of the heat kernel estimate in \cite{Zhang:Heat kernel} and harmonic analysis outside of a convex obstacle in \cite{KVZ16:IMRN Riesz tf}, we can define the homogeneous Sobolev space outside a ball.
\begin{definition}
    For $s>0$ and $1<p<\infty$, $\dot{H}^{s,p}(\Omega)$ denote the completions of $C_c^\infty(\Omega)$ under the norm 
    \[\|u\|_{\dot{H}^{s,p}(\Omega)}:=\|(-\Delta_\Omega)^\frac{s}{2}u\|_{L^p(\Omega)}.\]
    When $p=2$, we write $\dot{H}^s_D(\Omega)$ for $\dot{H}^{s,2}_D(\Omega)$, respectively.
\end{definition}
\begin{definition}
    The space $\dot{H}^{s,p}_{00}\Omega)$ is the completion of $C_c^\infty(\Omega)$ in $\dot{H}^{s,p}(\mathbb{R}^3)$.
\end{definition}

 \begin{proposition}[\cite{KVZ16:IMRN Riesz tf}]\label{prop:dense}
	For $1<p<\infty$, and $s<1+\frac{1}{p}$, $C^{\infty}_{c}(\Omega)$ is dense in $\dot H^{s,p}_{D}(\Omega)$.
\end{proposition}

The equivalence between $\dot H^{s,p}_{D}(\Omega) $ and $\dot H^{s,p}_{00}(\Omega)$ with proper exponents then follows from the square function estimate, Riesz potential estimate and the Hardy's inequalities for the exterior domain $\Omega$ and the whole space $\R^3$ in \cite{KVZ16:IMRN Riesz tf, MuscaluS:book1, Stein:book:SI}
\begin{proposition}
    $($\cite{KVZ16:IMRN Riesz tf}$)$. Let $\Omega=\mathbb{R}^3\setminus \overline{B}(0,1)$. Suppose $1<p<\infty$ and $0\leqslant s<\min\{1+\frac{1}{p},\frac{3}{p}\}$, then we have
    \begin{equation}
        \|(-\Delta_\Omega) ^\frac{s}{2}u\|_{L^p(\Omega)}\approx_{d,p,s} \|(-\Delta_{\mathbb{R}^3})^\frac{s}{2}u\|_{L^p(\Omega)}
    \end{equation}
    Thus $\|u\|_{\dot{H}^s_D(\Omega)}=\|u\|_{\dot{H}^s_{0}(\mathbb{R}^3)}$ for these values of the parameters.
\end{proposition}
As an direct application, we have the following fractional chain rule for the exterior domain.
\begin{corollary}[Fractional chain rule, \cite{KVZ16:IMRN Riesz tf}]\label{cor:fcr}
Suppose $F\in C^1(\mathbb{C})$, $s\in (0,1]$, and $1<p,p_1,p_2<\infty$ are such that  $\frac{1}{p}=\frac{1}{p_1}+\frac{1}{p_2}$ and $0<s<\min\left( 1+\frac{1}{p_2},   \frac{3}{p_2}\right) $ , then we have
    \begin{equation}
        \|(-\Delta_\Omega)^\frac{s}{2}F(u)\|_{L^p(\Omega)}\lesssim \|F'(u)\|_{L^{p_1}(\Omega)}\|(-\Delta_\Omega)^\frac{s}{2}u\|_{L^{p_2}(\Omega)}.
    \end{equation}
\end{corollary}

As stated in the introduction, we also need the following radial Sobolev inequality to control the growth of the energy and $L^2(\Omega)$ of the solution $v$ for the difference equation with low-frequency-localized data because of the super-conformal nonlinearity in next section. 
\begin{proposition}[Radial Sobolev inequality]\label{prop:rad sob est}
Fixed $1\leqslant p<\infty$. For a radial function $u\in \dot{H}^1_{D}(\Omega)\cap L^p(\Omega)$, then there exists a constant $C$ such that
    \begin{equation}
        r^\frac{4}{p+2}|u(r)|\leq C \|u\|_{L^p(\Omega)}^{\frac{p}{p+2}} \|u\|_{\dot H^1_{D, rad}(\Omega)}^{\frac{2}{p+2}}
    \end{equation}
    for almost everywhere $r\geq 1$. Particularly, we have $\dot{H}^1_{D,rad}(\Omega)\subset L^{\infty}(\Omega)$.
\end{proposition}
\begin{proof} We follow the similar argument as that in Radial Lemma in \cite{Strauss77:NLS:Soliton}. By the density result in Proposition \ref{prop:dense}, it suffices to show the inequality for the radial function $u\in C^{\infty}_{c}(\Omega)$. By the fundamental theorem of calculus, we have
	\begin{align*}
	\left(r^2u(r)^{\frac{p+2}{2}}\right)_r = 2r u(r)^{\frac{p+2}{2}} + \frac{p+2}{2}r^2u(r)^{\frac{p}{2}} u_r.
	\end{align*}	
Integrating on $(1,r)$, we have
		\begin{align*}
\left|	r^2u(r)^{\frac{p+2}{2}} \right| = & \left| \int^{r}_{1}2 s u(s)^{\frac{p+2}{2}} + \frac{p+2}{2}s^2u(s)^{\frac{p}{2}} u_s \; ds \right| \\
\lesssim &  \left( \int^r_1 |u|^p\, s^2\; ds \right)^{\frac{1}{2}}  \left( \int^r_1  \frac{|u|^2}{s^2}\, s^2 \; ds \right)^{\frac{1}{2}} +  \left( \int^r_1 |u|^p\, s^2\; ds \right)^{\frac{1}{2}}  \left( \int^r_1  |u_s|^2\, s^2 \; ds \right)^{\frac{1}{2}} \\
\lesssim &  \; \|u\|_{L^p(\Omega)}^{\frac{p}{2}}\|u\|_{\dot H^1_{D, rad}(\Omega)},
	\end{align*}
	where we used the Hardy inequality in \cite{KVZ16:IMRN Riesz tf} in the last inequality. This completes the proof.
	\end{proof}

\subsection{Littlewood-Paley theory}
In this subsection, we describe the Littlewood-Paley theory adapted to the Dirichlet-Laplacian operator $-\Delta_\Omega$ by the distorted Fourier transform.  We can also refer to \cite{GermHW14:PNLS nonlinear est, KriST:NLW bup, KriST:NLS, Sch:LW-P theory} for other applications.

Fixed $\phi: [0,\infty)\rightarrow [0,1]$ a smooth non-negative function satisfying $\phi(\lambda)=1$, for $0\leqslant\lambda\leqslant 1$ and $\phi(\lambda)=0$, for $\lambda\geqslant 2$. For each dyadic number $N\in 2^\mathbb{Z}$, we define
\[\phi_N(\lambda):=\phi(\frac{\lambda}{N})\ \ \    \text{and} \ \ \ \psi_N(\lambda):=\phi_N(\lambda)-\phi_\frac{N}{2}(\lambda).\] 
Let $m(\lambda)=\psi_N(\lambda)$ in (\ref{m}), we define the Littlewood-Paley projections
\begin{align*}
P^{\Omega}_{\leq N}u : = \phi_N(\sqrt{-\Delta_{\Omega}})u , \quad  P^{\Omega}_{N}u:= \psi_N(\sqrt{-\Delta_{\Omega}})u, \quad  P^{\Omega}_{>N}u:= I - P^{\Omega}_{\leq N} u.
\end{align*}

We define the homogeneous Besov space  as the following.
\begin{definition}
	Let $s\in \mathbb{R}$ and $1\leqslant q,r\leqslant\infty$. The homogeneous Besov space $\dot{B}^s_{D,q,r}(\Omega)$ consists of the distributions $u$ supported on $\Omega$ such that
	\begin{equation*}
	\|u\|_{\dot{B}^s_{D,q,r}(\Omega)}=\left(\sum\limits_{N\in 2^\mathbb{Z}}N^{sr}\|P_N^\Omega u\|^r_{L^q(\Omega)}\right)^\frac{1}{r}<\infty.
	\end{equation*}
\end{definition}

 As in the whole space case, we have the following Bernstein inequality.
\begin{proposition}[\cite{KVZ16:IMRN Riesz tf, LiSZ:NLS}]\label{prop:berns ests}
  	For any radial function $u\in C^{\infty}_{c}(\Omega)$, we have
  \begin{align}
  \left\|P^{\Omega}_{\leq N}u \right\|_{L^{p}(\Omega)}  +  &	\left\|P^{\Omega}_{ N}u \right\|_{L^{p}(\Omega)}  \lesssim  	\left\| u \right\|_{L^{p}(\Omega)}, \label{est:LP bd1} \\
  N^s	\left\|P^{\Omega}_{N} u \right\|_{L^{p}(\Omega)}  & \approx  \left\|\big(-\Delta_{\Omega}\big)^{s/2}P^{\Omega}_{N} u \right\|_{L^{p}(\Omega)} \label{est:LP bd2}
  \end{align}
  for any $1\leq p \leq \infty$ and $s\in \R$, Moreover, we have
  \begin{align*}
  \left\|P^{\Omega}_{\leq N} u \right\|_{L^{q}(\Omega)}  + 
  \left\|P^{\Omega}_{N} u \right\|_{L^{q}(\Omega)} \lesssim N^{3\big(\frac{1}{p} -\frac{1}{q}\big)} 	\left\| u \right\|_{L^{p}(\Omega)} 
  \end{align*}
  for all $1\leq p\leq q\leq \infty$. The implicit constants depend only on $p, q$ and $s$.
\end{proposition}
Therefore, for any $1<p<\infty$ and any radial $u\in L^p(\Omega)$, we have  the following homogeneous decomposition
\begin{align*}
u(x) = \sum_{N\in 2^{\Z}}  P^{\Omega}_{N} \, u(x).
\end{align*}
In particular, the sums converge in $L^p(\Omega)$.
    
\subsection{Strichartz estimate} 
Now we consider the 3D Dirichlet boundary value problem of the linear wave equation with radial data
\begin{equation}\label{eq:line wave}
\begin{cases}
\partial^2_t u -\Delta u = F, & (t,x)\in\R \times\Omega, \\
u(0,x)=u_0(x),  & x\in \Omega, \\
\partial_t u(0, x)= u_1(x), & x\in \Omega, \\
u(t,x)=0, & x\in \partial \Omega, 
\end{cases}
\end{equation}
where $\Omega =\R^3\backslash \bar{B}(0, 1)$ and $u_0, u_1$ and the inhomogeneous term $F$ are radial in $x$.  By the funcional calculus, we have
\begin{equation*}
u(t,x)=\cos(t\sqrt{-\Delta_{\Omega}})u_0+\frac{\sin(t\sqrt{-\Delta_{\Omega}})}{\sqrt{-\Delta_{\Omega}}} u_1 + \int^{t}_{0} \frac{\sin((t-s)\sqrt{-\Delta_{\Omega}})}{\sqrt{-\Delta_{\Omega}}} F(s, x)\, ds.
\end{equation*}

Let us denote the half-wave operator as 
$
U(t) = e^{it\sqrt{-\Delta_{\Omega}}},
$
then 
\begin{align*}
\cos(t\sqrt{-\Delta_{\Omega}})u_0 = 
\frac{U(t)+U(-t)}{2} u_0, \quad 
\frac{\sin(t\sqrt{-\Delta_{\Omega}})}{\sqrt{-\Delta_{\Omega}}} u_1 = \frac{U(t)-U(-t)}{2i\sqrt{-\Delta_{\Omega}}} u_1.
\end{align*}

By the distorted Fourier transform and the stationary phase estimate in \cite{XuYang:NLW}, we can obtain the following uniform dispersive estimate for linear wave equation \eqref{eq:line wave}, 
\begin{proposition}[Dispersive estimate. \cite{XuYang:NLW}]\label{prop:unif disp est}
	Let $2 \leq r \leq \infty$,  and the  radial function $f$ is supported on $\Omega$, then
	\begin{align*}
	\big\| U(t)  f \big\|_{\dot B^{-\beta(r)}_{D, r,2} (\Omega)} \lesssim |t|^{-\gamma(r)}  	\big\| f \big\|_{\dot B^{\beta(r)}_{D, r',2} (\Omega)} 
	\end{align*}
	where $\beta(r)=\gamma(r)=1- \frac2r$.
\end{proposition}

By the dual argument in \cite{GinV:NLW:St est}, and Christ-Kiselev lemma in \cite{ChrstK:lem}, see also \cite{KeelTao:Endpoint}, we have 
 the following global-in-time Strichartz estimate.

\begin{proposition}[Strichartz estimate, \cite{XuYang:NLW}]\label{prop: ST est} Let $\rho_1, \rho_2, \mu\in \mathbb{R}$ and $2\leq q_1, q_2, r_1, r_2\leq \infty$ and let the following conditions be satisfied
	\begin{align*}
	0\,\leq \, \frac{1}{q_i} +\frac{1}{r_i} \leq \,\frac12, \quad r_i \not = \infty, \quad  i=1,\, 2
	\end{align*}
	\begin{align*}
	\rho_1 +3\left(\frac12 -\frac{1}{r_1}\right)-\frac{1}{q_1} = \mu, \quad 
	\rho_2  +3\left(\frac12 -\frac{1}{r_2}\right)-\frac{1}{q_2} = 1-\mu 
	\end{align*} 
Let $u_0, u_1$ and $F$ be radial in $x$, and $u: I\times \Omega\longrightarrow \mathbb{R}$ be a solution to linear wave equation \eqref{eq:line wave} with $0\in I$. Then $u$ satisfies the estimate
	\begin{align*}
\left\|u\right\|_{L^{q_1}_I\dot B^{\rho_1}_{D, r_1, 2}(\Omega)\cap C(I; \dot H_D^{\mu}(\Omega))} + 	\left\|\partial_t u\right\|&_{L^{q_1}_I\dot B^{\rho_1-1}_{D, r_1, 2}(\Omega) \cap C(I; \dot H_D^{\mu-1}(\Omega)) }   \\
& \lesssim \; \|(u_0, u_1)\|_{\dot H^{\mu}_{D, rad}(\Omega)\times \dot H^{\mu-1}_{D,rad}(\Omega)} + \|F\|_{L^{q'_2}_I\dot B^{-\rho_2}_{D, r'_2, 2}(\Omega)}.
\end{align*}
\end{proposition}
For the global-in-time Strichartz estimate in the nonradial case, we can also refer to \cite{DuyYang24:LW Str}.

By the refined (localized) dispersive estimate instead of  the uniform dispersive estimate in Proposition \ref{prop:unif disp est}, we can obtain the improved global-in-time endpoint Strichartz estimates.
\begin{proposition}[Endpoint Strichartz estimate with radial data, \cite{XuYang:NLW}]\label{prop: end ST est} Let $u_0, u_1$ and $F$ be radial in $x$ variable, and $u: I\times \Omega\longrightarrow \mathbb{R}$ be  a solution to linear wave equation \eqref{eq:line wave} with $0\in I$. If $q>4$ and $s  = 1-3/q, $
	then we have
	\begin{align} \label{ese}
	\left\|u\right\|_{L^2_tL^q_x(I\times \Omega)}  \lesssim \|(u_0,u_1)\|_{\dot H^s_{D, rad}(\Omega)\times \dot H^{s-1}_D(\Omega)} + \|F\|_{L^1_t \dot H^{s-1}_{D, rad}(I\times\Omega)}. 
	\end{align}
\end{proposition}

\section{Proof of Theorem \ref{thm:GWP}} \label{sect:pf rest}
In this section, we will combine the Fourier truncation method in \cite{Bourgain98:FTM},  energy method, the endpoint Strichartz estimates with radial Sobolev inequality in Section \ref{sect:DFT and ST est}  to prove Theorem \ref{thm:GWP}. 

Let $s_c=\frac{3}{2}-\frac{2}{p-1}<s<1$ and  $(u_0,u_1)\in (\dot{H}^s_{D, rad}(\Omega)\cap L^{p+1}(\Omega))\times \dot{H}^{s-1}_{D, rad}(\Omega).$

Now we decompose initial data of \eqref{eq:pnlw} as the following \[u_i=w_i+v_i,\]
where $i=0,1$ and $w_i=P^{\Omega}_{\geqslant 2^J}u_i, v_i=P^{\Omega}_{<2^J}u_i$, $J$ is chosen sufficiently large later. 

\subsection{Global well-posedness of \eqref{eq:pnlw} with high frequency data}\label{subsect:w}
Consider the equation 
\begin{equation}\label{eq:hfreqp}
    \partial_t^2 w-\Delta w+|w|^{p-1}w=0
\end{equation}
with intial data $(w_0,w_1)\in \dot{H}^{s}_D(\Omega)\times \dot{H}^{s-1}_D(\Omega)$ and the zero Dirichlet boundary value $w(t,x)=0$, $x\in \partial\Omega$. For an arbitrary small number $\varepsilon>0$, if we choose a constant $J=J(\varepsilon)>0$ sufficiently large such that $$\|(w_0,w_1)\|_{\dot{H}^s_{D,rad}(\Omega)\times\dot{H}^{s-1}_{D,rad}(\Omega)}\lesssim \varepsilon.$$
By the Bernstein estimate in Proposition \ref{prop:berns ests},  we have \begin{align*}\|(w_0,w_1)\|_{\dot{H}^{s_c}_{D, rad}(\Omega)\times\dot{H}^{s_c-1}_{D, rad}(\Omega)}\lesssim 2^{J(s_c-s)}\|(u_0,u_1)\|_{\dot{H}^s_{D, rad}(\Omega)\times\dot{H}^{s-1}_{D, rad}(\Omega)}\lesssim 2^{J(s_c-s)}.
\end{align*}
Using the standard well-posedness theory in \cite{Sogge95:book, XuYang:NLW}, we can obtain the following global result.

\begin{proposition}\label{prop:gwp hp}
    Let $s_c=\frac{3}{2}-\frac{2}{p-1}<s\leqslant 1$, there exists a small number $\varepsilon_0>0$ and a large constant $J=J(\varepsilon)>0$ such that $$2^{J(s_c-s)}  \lesssim \varepsilon_0,$$
    then the equation \eqref{eq:hfreqp} is global well-posedness in $\dot{H}^{s_c}_{D, rad}(\Omega)\cap \dot{H}^s_{D, rad}(\Omega)$. Moreover, we have the following estimates
    \begin{equation}\label{w:Hs est}
        \|w\|_{L^\infty_t\left(\mathbb{R};\dot{H}^s_{D, rad}(\Omega)\right)}\lesssim \varepsilon,
    \end{equation}
    \begin{equation}\label{w:Hsc est}
        \|w\|_{L_t^{\frac{2p}{1+s_c}}L_x^{\frac{2p}{2-s_c}}(\mathbb{R}\times \Omega)}+\|w\|_{L_{t,x}^{2(p-1)}(\mathbb{R}\times \Omega)}+\|w\|_{L^{p-1}_t L^{3(p-1)}_x(\mathbb{R}\times\Omega)}+\|w\|_{L^2_t L^\frac{3}{1-s_c}_x(\mathbb{R}\times\Omega)}\lesssim 2^{J(s_c-s)}.
    \end{equation}
\end{proposition}

\begin{proof}
    We firstly prove the existence and uniqueness by the contraction mapping principle. We will show the following map $w\rightarrow \mathcal{T}(w)$ defined by the following formula
    \begin{equation}\label{map}
        \mathcal{T}(w)=\cos(t\sqrt{-\Delta_\Omega})w_0+\frac{\sin(t\sqrt{-\Delta_\Omega})}{\sqrt{-\Delta_\Omega}}w_1-\int_0^t \frac{\sin((t-s)\sqrt{-\Delta_\Omega})}{\sqrt{-\Delta_\Omega}}\left(|w|^{p-1}w\right)(s)\,ds,
    \end{equation}
    is a contraction map on the set 
    \begin{align*}
    X:=\{w\in C^0_t\left(\mathbb{\R};\dot{H}^{s_c}_{D, rad}(\Omega) \right) \cap L_t^{\frac{2p}{1+s_c}}L_x^{\frac{2p}{2-s_c}}(\mathbb{R}\times \Omega): & \\  \|w\|_{ L_t^\infty \dot{H}^{s_c}_{x,D}(\mathbb{R}\times \Omega)\cap L_t^{\frac{2p}{1+s_c}}L_x^{\frac{2p}{2-s_c}}(\mathbb{R}\times \Omega)}
   & \leqslant 2C\|(w_0,w_1)\|_{\dot{H}^{s_c}_D(\Omega)\times \dot{H}^{{s_c}-1}_D(\Omega)}\},
    \end{align*}
    with the metric $d(w_1,w_2)=\|w_1-w_2\|_{L_t^{\frac{2p}{1+s_c}}L_x^{\frac{2p}{2-s_c}}(\mathbb{R}\times\Omega)}$, and $C$ denotes the constant from the Strichartz estimate in Proposition \ref{prop: ST est}. 
    
   By the Strichartz estimate in Proposition \eqref{prop: ST est} and the H\"older inequality, we have  for any $w\in X$ that 
    \begin{align}
        \|\mathcal{T}(w)\|_{L_t^{\frac{2p}{1+s_c}}L_x^{\frac{2p}{2-s_c}}(\mathbb{R}\times \Omega)}\lesssim & \|(w_0,w_1)\|_{\dot{H}^{s_c}_D(\Omega)\times \dot{H}^{{s_c}-1}_D(\Omega)}+\||w|^{p-1}w\|_{L_t^{\frac{2}{1+s_c}}L_x^{\frac{2}{2-s_c}}(\mathbb{R}\times \Omega)} \nonumber\\
        \lesssim & \|(w_0,w_1)\|_{\dot{H}^{s_c}_D(\Omega)\times \dot{H}^{{s_c}-1}_D(\Omega)}+\|w\|^p_{L_t^{\frac{2p}{1+s_c}}L_x^{\frac{2p}{2-s_c}}(\mathbb{R}\times \Omega)}, \label{est:st pair}
    \end{align}
   and arguing as above, we obtain for any $w_1,w_2\in X$ that
    \begin{align}\label{est:contract}
        \|\mathcal{T}(w_1)-& \mathcal{T}(w_2)\|_{L_t^{\frac{2p}{1+s_c}}L_x^{\frac{2p}{2-s_c}}(\mathbb{R}\times \Omega)}\nonumber\\
        =& \left\|\int_0^t \frac{\sin((t-s)\sqrt{-\Delta_\Omega})}{\sqrt{-\Delta_\Omega}}\left(|w_1|^{p-1}w_1-|w_2|^{p-1}w_2\right)(s)\,ds\right\|_{L_t^{\frac{2p}{1+s_c}}L_x^{\frac{2p}{2-s_c}}(\mathbb{R}\times \Omega)}\nonumber\\
        \lesssim& \||w_1|^{p-1}w_1-|w_2|^{p-1}w_2\|_{L_t^{\frac{2}{1+s_c}}L_x^{\frac{2}{2-s_c}}(\mathbb{R}\times \Omega)}\nonumber\\
        \lesssim& \|w_1-w_2\|_{L_t^{\frac{2p}{1+s_c}}L_x^{\frac{2p}{2-s_c}}(\mathbb{R}\times \Omega)}\left(\|w_1\|_{L_t^{\frac{2p}{1+s_c}}L_x^{\frac{2p}{2-s_c}}(\mathbb{R}\times \Omega)}^{p-1}+\|w_2\|_{L_t^{\frac{2p}{1+s_c}}L_x^{\frac{2p}{2-s_c}}(\mathbb{R}\times \Omega)}^{p-1}\right).
    \end{align}
    	Thus, choosing $J=J(\epsilon)$ large enough (if necessary), we can guarantee that $\mathcal{T}$ maps the set $X$  into itself and is a contraction map on the set $X$. By the contraction mapping theorem, it follows that $\mathcal{T}$ has a fixed point $w$ in $X$.
    

 Secondly,  by the Strichartz estimate,  the H\"older inequality and the fractional chain rule in Corollary \ref{cor:fcr}, we obtain
 \begin{align*}
 \|w\|_{L_{t,x}^{2(p-1)}(\mathbb{R}\times \Omega)}+ \|w\|_{L^{p-1}_t L^{3(p-1)}_x(\mathbb{R}\times\Omega)} \lesssim& \|(w_0,w_1)\|_{\dot{H}^{s_c}_D(\Omega)\times \dot{H}^{{s_c}-1}_D(\Omega)}+\|w\|_{L_t^{\frac{2p}{1+s_c}}L_x^{\frac{2p}{2-s_c}}(\mathbb{R}\times \Omega)}^p \\
 \lesssim& \|(w_0,w_1)\|_{\dot{H}^{s_c}_D(\Omega)\times \dot{H}^{{s_c}-1}_D(\Omega)}
 \end{align*}
 and 
    \begin{align*}
        \|w&\|_{L^\infty_t\left(\mathbb{R}; \dot{H}^s_{D,rad}( \Omega)\right)\cap  L^{2(p-1)}_t\left(\mathbb{R}; \dot{H}^{s-{s_c},2(p-1)}_{D,rad}( \Omega)\right) \cap L^{\frac{2p}{1+s_c}}_t\left(\mathbb{R}; \dot{H}^{s-{s_c},\frac{2p}{2-s_c}}_{D,rad}(\Omega)\right)}\\
       &\qquad \lesssim \|(w_0,w_1)\|_{\dot{H}^s_D(\Omega)\times \dot{H}^{s-1}_D(\Omega)}+\||w|^{p-1}w\|_{L^{\frac{2}{1+s_c}}_t\left(\mathbb{R}; \dot{H}^{s-{s_c},\frac{2}{2-s_c}}_{D,rad}(\Omega)\right)}\\ 
      & \qquad \lesssim \|(w_0,w_1)\|_{\dot{H}^s_D(\Omega)\times \dot{H}^{s-1}_D(\Omega)}+\|w\|^{p-1}_{L_t^{\frac{2p}{1+s_c}}L_x^{\frac{2p}{2-s_c}}(\mathbb{R}\times \Omega)} \|w\|_{L^{\frac{2p}{1+s_c}}_t\left(\mathbb{R}; \dot{H}^{s-{s_c},\frac{2p}{2-s_c}}_{D,rad}(\Omega)\right)}\\
      &\qquad  \lesssim \|(w_0,w_1)\|_{\dot{H}^s_D(\Omega)\times \dot{H}^{s-1}_D(\Omega)}.
    \end{align*}
    
Finally, by the endpoint Strichartz estimate in Proposition \ref{prop: end ST est} and the Sobolev inequality $L^\frac{3(p-1)}{p+1}(\Omega)\hookrightarrow\dot{H}^{s_c-1}_{D}(\Omega)$, we also have
    \begin{align}
        \|w\|_{L^2_t L^\frac{3}{1-s_c}_x(\mathbb{R}\times\Omega)}&\lesssim \|(w_0,w_1)\|_{\dot{H}^{s_c}_{D, rad}(\Omega)\times \dot{H}^{s_c-1}_{D, rad}(\Omega)}+\||w|^{p-1}w\|_{L^1_t\left(\mathbb{R};\dot{H}^{s_c-1}_{D, rad}(\Omega)\right)}\nonumber\\
        &\lesssim \|(w_0,w_1)\|_{\dot{H}^{s_c}_{D, rad}(\Omega)\times \dot{H}^{s_c-1}_{D, rad}(\Omega)}+\|w^{p-1}w\|_{L^1_tL_x^\frac{3(p-1)}{p+1}(\mathbb{R}\times\Omega)}\nonumber\\
        &\lesssim  \|(w_0,w_1)\|_{\dot{H}^{s_c}_{D,rad}(\Omega)\times \dot{H}^{s_c-1}_{D,rad}(\Omega)}+\|w\|^{p-1}_{L^{2(p-1)}_{t,x}(\mathbb{R}\times \Omega)}\|w\|_{L^2_t L^\frac{3}{1-s_c}_x(\mathbb{R}\times\Omega)}\nonumber\\
        &\lesssim \|(w_0,w_1)\|_{\dot{H}^{s_c}_{D,rad}(\Omega)\times \dot{H}^{s_c-1}_{D,rad}(\Omega)}. 
    \end{align}
then we can obtain the result.
\end{proof}

\subsection{Local-in-time energy analysis for low frequency part}\label{subs:v H1 LWP}
Let $w$ be the small solution of \eqref{eq:hfreqp} in Proposition \ref{prop:gwp hp},   we now consider the following difference equation
\begin{equation}\label{eq:lfreqp}
    \partial_t^2 v-\Delta v+F(v,w)=0
\end{equation} where $F(v,w)=|v+w|^{p-1}(v+w)-|w|^{p-1}w$ and  initial data $(v_0,v_1)\in \left( \dot H^{1}_{D, rad}(\Omega) \cap L^{p+1}_{rad}(\Omega) \right)\times  L^2_{rad}(\Omega) $ and the zero Dirichlet boundary value $v(t,x)=0$, $x\in \partial\Omega$.

By the Taylor formula, we can obtain the following explicit expression  
\begin{align}
    F(v,w)-|v|^{p-1}v
    &=\mathcal{O}(|v|^{p-1}|w|)+\mathcal{O}(|v||w|^{p-1}). \label{term: F}
\end{align}

Firstly, by the contraction mapping argument. we can obtain local well-posedness of \eqref{eq:lfreqp} in the energy space.

\begin{proposition}\label{prop:lwp lp}
    Let $w$ be the solution of \eqref{eq:hfreqp} in Proposition \ref{prop:gwp hp}, and $T=T\left(\|(v_0, v_1)\|_{\dot{H}^1_D(\Omega)\times L^2(\Omega)}\right)$ such that
    \begin{align*}
    T\cdot \|(v_0,v_1)\|^{p-1}_{\dot{H}^1_D(\Omega)\times L^2(\Omega)}\lesssim 1,
    \end{align*} then there exists a unique solution $v\in C\left([0, T);\dot{H}^1_{D}(\Omega)\right)$ of \eqref{eq:lfreqp}. Moreover, we have the estimate
    \begin{equation}
        \|v\|_{L_t^\infty\left([0, T); \dot{H}^1_{D}(\Omega)\right)}\lesssim \|(v_0,v_1)\|_{\dot{H}^1_D(\Omega)\times L^2(\Omega)}.
    \end{equation}
\end{proposition}
\begin{proof}
 Let $T>0$ be small determined later,  it suffices to show that the following map $\mathcal{T}$ 
    \begin{equation}\label{map}
        \mathcal{T}(v)=\cos(t\sqrt{-\Delta_\Omega})v_0+\frac{\sin(t\sqrt{-\Delta_\Omega})}{\sqrt{-\Delta_\Omega}}v_1-\int_0^t \frac{\sin((t-s)\sqrt{-\Delta_\Omega})}{\sqrt{-\Delta_\Omega}}F(v,w)(s)\,ds
    \end{equation}
 is a contraction map on the set
    \begin{align*}
    Y:=\left\{v\in C\left( [0, T); \dot{H}^1_{x,D}(\Omega)\right): \|v\|_{L_t^\infty\dot{H}^1_{x,D}((0,T]\times\Omega)}\leqslant 2C\|(v_0,v_1)\|_{\dot{H}^1_D(\Omega)\times L^2(\Omega)}\right\}
    \end{align*}
equipped with the metric $d(v_1,v_2)=\|v_1-v_2\|_{L^\infty\left([0, T); \dot{H}^1_{D}(\Omega)\right)}$, where $C$ denotes the constant from the Strichartz estimate.

By \eqref{term: F} and  the Strichartz estimate in Proposition \ref{prop: ST est}, we have
\begin{align}
            \|\mathcal{T}(v)\|_{L_t^\infty\dot{H}^1_{x,D}((0,T]\times\Omega)}&\lesssim \|(v_0,v_1)\|_{\dot{H}^1_D(\Omega)\times L^2(\Omega)}+\|F(v,w)\|_{L_t^1L_x^2((0,T]\times\Omega)} \nonumber\\
            &= C \|(v_0,v_1)\|_{\dot{H}^1_D(\Omega)\times L^2(\Omega)}+\text{\uppercase\expandafter{\romannumeral1}}+\text{\uppercase\expandafter{\romannumeral2}}+\text{\uppercase\expandafter{\romannumeral3}} \label{est: energy lwp}
\end{align}

By the H\"older inequality,  the Sobolev inequality and  the radial Sobolev inequality in Proposition \ref{prop:rad sob est}, we have
\begin{align}
    \text{\uppercase\expandafter{\romannumeral1}}
    \lesssim  \int_0^T\||v|^{p-1}v\|_{L^{2}(\Omega)}\,dt\lesssim& \int_0^T\|v(t)\|_{L^6(\Omega)}^3\|v(t)\|^{p-3}_{L^\infty(\Omega)}\,dt   \nonumber \\
         \lesssim& \int_0^T\|v(t)\|^{3}_{L^6(\Omega)} \|v(t)\|^{p-3}_{\dot H^1_{D}(\Omega)}\,dt 
    \lesssim\; T\, \|v\|^p_{L^\infty\left([0, T);\dot{H}^1_{D}(\Omega)\right)}.  \label{term: est1}
\end{align}
By the analogue estimates, we have
\begin{align}\label{term: est2}
    \text{\uppercase\expandafter{\romannumeral2}}
    \lesssim \int_0^T\||v|^{p-1}|w|\|_{L^2(\Omega)}\,dt &\lesssim\int_0^T\|v(t)\|_{L^{\frac{3(p-1)^2}{2(p-2)}}(\Omega)}^{p-1}  \|w(t)\|_{L^{\frac{3}{1-s_c}}(\Omega)}\,dt \nonumber\\
    &\lesssim\int_0^T\|v(t)\|^{\frac{4(p-2)}{p-1}}_{L^6(\Omega)} \|v(t)\|^{(p-1)-\frac{4(p-2)}{p-1}}_{L^{\infty}(\Omega)} \|w(t)\|_{L^{\frac{3}{1-s_c}}(\Omega)}\,dt \nonumber\\
    &\lesssim T^\frac{1}{2}\, \|w\|_{L_t^2L_x^{\frac{3}{1-s_c}}([0,T)\times\Omega)}\|v\|_{L_t^\infty\left([0, T); \dot{H}^1_{D}(\Omega)\right)}^{p-1},
\end{align}
and
\begin{align}\label{term: est3}
        \text{\uppercase\expandafter{\romannumeral3}}
        \lesssim \int_0^T \||v||w|^{p-1}\|_{L^2(\Omega)}\,dt
        &\lesssim \int_0^T\|v(t)\|_{L^6(\Omega)}\|w(t)\|^{p-1}_{L^{3(p-1)}(\Omega)}\,dt\nonumber\\
        &\lesssim \|w\|_{L_t^{p-1}L_x^{3(p-1)}([0,T)\times\Omega)}^{p-1}\|v\|_{L_t^\infty\left([0, T);\dot{H}^1_{D}(\Omega)\right)} ,
\end{align}
Insertting \eqref{term: est1},\eqref{term: est2}, and \eqref{term: est3} into \eqref{est: energy lwp}, we obtain
 \begin{align}
\|\mathcal{T}(v)\|_{L_t^\infty\left([0, T); \dot{H}^1_{D}(\Omega)\right)}
 \lesssim &  \|(v_0,v_1)\|_{\dot{H}^1_D(\Omega)\times L^2(\Omega)} 
  + 
T\, \|v\|^p_{L^\infty\left([0, T);\dot{H}^1_{D}(\Omega)\right)} \nonumber \\
&  +  T^\frac{1}{2}\, \|w\|_{L_t^2L_x^{\frac{3}{1-s_c}}([0,T)\times\Omega)}\|v\|_{L_t^\infty\left([0, T); \dot{H}^1_{D}(\Omega)\right)}^{p-1}  \nonumber \\
& +  \|w\|_{L_t^{p-1}L_x^{3(p-1)}([0,T)\times\Omega)}^{p-1}\|v\|_{L_t^\infty\left([0, T);\dot{H}^1_{D}(\Omega)\right)} . \label{est: energy lwp2}
\end{align}

Furthermore, by the Strichartz estimate in Proposition \eqref{prop: ST est}, we have
\begin{align*}
    \|\mathcal{T}(v_1)-\mathcal{T}(v_2)\|_{L_t^\infty\left([0, T);\dot{H}^1_{D}(\Omega)\right)}\lesssim \|F(v_1,w)-F(v_2,w)\|_{L_t^1L_x^2((0,T]\times \Omega)}
    \lesssim \text{\uppercase\expandafter{\romannumeral1}}'+\text{\uppercase\expandafter{\romannumeral2}}'+\text{\uppercase\expandafter{\romannumeral3}}'. 
\end{align*}
We estimate one by one as follows.
\begin{align*}
    \text{\uppercase\expandafter{\romannumeral1}}'&\lesssim \int_0^T\left\||v_1|^{p-1}v_1-|v_2|^{p-1}v_2\right\|_{L^{2}(\Omega)}\,dt \\
   &\lesssim \int_0^T\|v_1(t)-v_2(t)\|_{L^6(\Omega)}\left(\|v_1(t)\|_{L^{3(p-1)}(\Omega)}^{p-1}+\|v_2(t)\|_{L^{3(p-1)}(\Omega)}^{p-1}\right)\,dt \\
   &\lesssim T\, \|v_1-v_2\|_{L_t^\infty\left([0, T); \dot{H}^1_{D}(\Omega)\right)}\left(\|v_1\|^{p-1}_{L_t^\infty\left([0, T); \dot{H}^1_{D}(\Omega)\right)}+\|v_2\|^{p-1}_{L_t^\infty\left([0, T); \dot{H}^1_{D}(\Omega)\right)}\right),
\end{align*}
and
\begin{align*}
   \text{\uppercase\expandafter{\romannumeral2}}'
   \lesssim&  \int_0^T\left\|\big(|v_1|^{p-1}-|v_2|^{p-1}\big)|w|\right\|_{L^2(\Omega)}\,dt \\
 \lesssim   & \int_0^T \|v_1(t)-v_2(t)\|_{L^6(\Omega)}\|w(t)\|_{L^\frac{3}{1-s_c}(\Omega)}(\|v_1(t)\|_{L^\frac{6(p-1)(p-2)}{3p-7}(\Omega)}^{p-2}+\|v_2(t)\|_{L^\frac{6(p-1)(p-2)}{3p-7}(\Omega)}^{p-2})\,dt  \\
\lesssim    &  T^\frac{1}{2}\|v_1-v_2\|_{L_t^\infty\left([0, T); \dot{H}^1_{D}(\Omega)\right)}\|w\|_{L_t^2L_x^\frac{3}{1-s_c}((0,T]\times\Omega)}\left(\|v_1\|^{p-2}_{L_t^\infty\left([0, T); \dot{H}^1_{D}(\Omega)\right)} + \|v_2\|^{p-2}_{L_t^\infty\left([0, T); \dot{H}^1_{D}(\Omega)\right)} \right),
\end{align*}
and
\begin{align*}
    \text{\uppercase\expandafter{\romannumeral3}}'\lesssim \int_0^T \left\| |v_1-v_2||w(t)|^{p-1}\right\|_{L^2(\Omega)}\,dt 
    &\lesssim \int_0^T\|v_1(t)-v_2(t)\|_{L^6(\Omega)}\|w(t)\|_{L^{3(p-1)}(\Omega)}^{p-1}\,dt \\
    &\lesssim \|v_1-v_2\|_{L_t^\infty\left([0, T); \dot{H}^1_{D}(\Omega)\right)}\|w\|_{L_t^{p-1}L_x^{3(p-1)}([0,T)\times\Omega)}^{p-1}, 
\end{align*}
	Thus, by  Proposition \ref{prop:gwp hp} and choosing $T$ such that $$2\; T \cdot  \left(2 \; C \left\| \big( v_0, v_1 \big )\right\|_{\dot H^{1}_{D, rad}(\Omega)  \times L^2(\Omega)}   \right)^{p-1}\leq \frac12, $$  we can guarantee that $\mathcal{T}$ maps the set $Y$ into itself and is a contraction map on the set $Y$. By the contraction mapping theorem, it follows that $\mathcal{T}$ has a fixed point $v$ in $Y$. 
\end{proof}

 \subsection{Global-in-time energy analysis for low frequency part} \label{subs: lfreq H1 growth control}
In this part, we extend local energy solution $v$ of the difference equation \eqref{eq:lfreqp} to global one. By Proposition \ref{prop:lwp lp}, it suffices to control  the growth of the energy of the solution $v$ of \eqref{eq:lfreqp}. Let us denote the energy of $v$ by
\begin{equation}
    E(v)(t)=\int_\Omega \frac{1}{2}|v_t(t)|^2+\frac{1}{2}|\nabla v(t)|^2+\frac{1}{p+1}|v(t)|^{p+1}\,dx,
\end{equation}
and set $E_T=\max\limits_{0\leqslant t\leqslant T}E(v)(t)$, where $[0,T)$ is the maximal lifespan interval in Proposition \ref{prop:lwp lp}. In fact, the energy is not a conserved because of the external source term $F(v,w)-|v|^{p-1}v$.

Now we give the control of the energy of $v$.
\begin{proposition}\label{prop: v energy est}
    Let $w$ be the global solution of \eqref{eq:hfreqp}, $v$ be the local energy solution of \eqref{eq:lfreqp} in Proposition \ref{prop:lwp lp}, and $[0,T)$ be the maximal existence interval of $v$, then for any $t\in [0,T)$, we have 
    \begin{align}
                E(v)(t)\leqslant E_T\lesssim E(v)(0)& +T^\frac{1}{2}E_T^\frac{9(p-1)}{2(p+3)}\|w\|_{L_t^2L_x^\frac{3}{1-s_c}([0,T)\times \Omega)} +E_T \|w\|^{p-1}_{L_t^{p-1}L_x^{3(p-1)}([0,T)\times \Omega)} 
    \end{align}
\end{proposition}

\begin{proof}
    Firstly, by \eqref{term: F},  we have the energy variation
    \begin{align}\label{term:dtE}
     \frac{d}{dt}E(t)&=\int_\Omega v_t\left(F(v,w)-|v|^{p-1}v\right)\, dx\nonumber\\
     &=-\int_\Omega v_t\left(\mathcal{O}(|v|^{p-1}|w|)+\mathcal{O}(|v||w|^{p-1})\right)\,dx.
    \end{align}
By the H\"older inequality and the radial Sobolev inequality in Proposition \ref{prop:rad sob est}, we have
\begin{align}\label{term:w1}
    \left|\int_0^T\int_\Omega v_t |v|^{p-1}|w|\,dx\,dt\right|&\lesssim \int_0^T \|v_t\|_{L^2(\Omega)}\|w\|_{L^\frac{3}{1-s_c}(\Omega)}\|v\|_{L^6(\Omega)}^\frac{4(p-2)}{p-1}\|v\|_{L^\infty(\Omega)}^\frac{(p-3)^2}{p-1}\; dt \nonumber\\
    &\lesssim \int_0^T \|v_t\|_{L^2(\Omega)}\|w\|_{L^\frac{3}{1-s_c}(\Omega)}\|v\|_{\dot{H}^1(\Omega)}^\frac{4(p-2)}{p-1}\||x|^{\frac{4}{p+3}}v\|_{L^\infty(\Omega)}^\frac{(p-3)^2}{p-1}\; dt\nonumber\\
     &\lesssim \int_0^T \|v_t\|_{L^2(\Omega)}\|w\|_{L^\frac{3}{1-s_c}(\Omega)}\|v\|_{\dot{H}^1(\Omega)}^\frac{4(p-2)}{p-1}\
     \|v\|^{\frac{(p+1)(p-3)^2}{(p-1)(p+3)}}_{L^{p+1}(\Omega)}\|v\|^{\frac{2(p-3)^2}{(p-1)(p+3)}}_{\dot H^1_D(\Omega)}
     \; dt\nonumber\\
    &\lesssim T^\frac{1}{2}\, E_T^{\frac{1}{2}+\frac{2(p-2)}{p-1}+\frac{2(p-3)^2}{(p-1)(p+3)}}\|w\|_{L_t^2L_x^\frac{3}{1-s_c}([0,T)\times \Omega)}\nonumber\\
    &\lesssim T^\frac{1}{2}\, E_T^\frac{9(p-1)}{2(p+3)}\|w\|_{L_t^2L_x^\frac{3}{1-s_c}([0,T)\times \Omega)},
\end{align}    
and
\begin{align}\label{term:w3}    
    \left|\int_0^T\int_\Omega v_t |v||w|^{p-1}\,dx\,dt\right|&\lesssim \int_0^T \|v_t\|_{L^2(\Omega)}\|w\|_{L^{3(p-1)}(\Omega)}^{p-1}\|v\|_{L^6(\Omega)}\; dt\nonumber\\
    &\lesssim E_T\, \|w\|^{p-1}_{L_t^{p-1}L_x^{3(p-1)}([0,T)\times \Omega)}.
\end{align}
Inserting \eqref{term:w1} and \eqref{term:w3} into \eqref{term:dtE}, we can obtain the result. 
\end{proof}

Now we can control the energy growth of $v$. Noticed that $(u_0, u_1)\in \left(\dot H^{s}_{D}(\Omega) \cap L^{p+1}(\Omega) \right)\times \dot H^{s-1}_{D}(\Omega)$, we have
\begin{align}\label{e0}
        E(v)(0)&\lesssim\|(v_0,v_1)\|^2_{\dot{H}^1_D(\Omega)\times L^2(\Omega)}+\|v_0\|^{p+1}_{L^{p+1}(\Omega)}\nonumber\\
        &\lesssim 2^{2J(1-s)}\|(v_0,v_1)\|_{\dot{H}^{s}_D(\Omega)\times \dot{H}^{s-1}_D(\Omega)}^2+\|v_0\|^{p+1}_{L^{p+1}(\Omega)} \lesssim 2^{2J(1-s)}. 
\end{align}
By Proposition \ref{prop:gwp hp} and Proposition \ref{prop: v energy est}, we have
\begin{equation}
    E(v)(t)\leqslant E_T\lesssim 2^{2J(1-s)}+T^\frac{1}{2}E_T^\frac{9(p-1)}{2(p+3)}2^{J(s_c-s)}+E_T2^{2J(s_c-s)}. 
\end{equation}
If we choose $T$ such that
\begin{equation}\label{est:T}
    T^\frac{1}{2}2^{2J(1-s)\frac{9(p-1)}{2(p+3)}}2^{J(s_c-s)}\approx 2^{2J(1-s)}\Longleftrightarrow  T\approx  2^{\frac{2J}{p+3} \left[4(2p-3)s -4(2p-3)+(p+3)(1-s_c)\right]},
\end{equation}
we obtain that
\begin{equation}\label{est:Energy}
    E_T\lesssim 2^{2J(1-s)}\approx 
    T^{\frac{(p+3)(1-s)}{4(2p-3)s-4(2p-3)+(p+3)(1-s_c)} },
\end{equation}
which shows that the energy of $v$ can be controlled globally as long as  $s>1-\frac{(p+3)(1-s_c)}{4(2p-3)}$. 


\subsection{Growth estimate of low frequency part in $\dot{H}^s_{D}(\Omega)$}
From Subsection \ref{subsect:w}, and Subsection \ref{subs: lfreq H1 growth control}, we know that the solution $u$ of \eqref{eq:pnlw} globally exists in  $\dot H^{s_c}_{D, rad}(\Omega) \cap \dot H^s_{D, rad}(\Omega) + \dot H^{1}_{D, rad}(\Omega) $ with  $s>1-\frac{(p+3)(1-s_c)}{4(2p-3)}$. Now we show
the $\dot H^s_D(\Omega)$-estimate of $u$ and complets the proof of Theorem \ref{thm:GWP}.

By Proposition \ref{prop:gwp hp}, it suffices to show the estimate of $v$ in $\dot H^s_D(\Omega)$. Noticed that the homogeneous part of $v$ 
\[\cos(t\sqrt{-\Delta_\Omega})v_0+\frac{\sin(t\sqrt{-\Delta_\Omega})}{\sqrt{-\Delta_\Omega}}v_1\]
is bounded in $\dot{H}^s_D(\Omega)$ by the energy estimate. 

It reduces to show the estimate of inhomogeneous part of $v$ in $\dot{H}^s_D(\Omega)$. 

Using Proposition \ref{prop: v energy est} and interpolation argument, it suffices to show the estimate of inhomogeneous part of $v$ in $L^2(\Omega)$. By the distorted Fourier transform, we have for any $0\leqslant t\leqslant T$,
\begin{align}
    &\left\|\int_0^t \frac{\sin((t-s)\sqrt{-\Delta_\Omega})}{\sqrt{-\Delta_\Omega}}F(v,w)(s)\,ds\right\|_{L^\infty_tL^2_x([0,T)\times \Omega)}\nonumber\\
    \lesssim &\int^t_0  \left \| \frac{\sin\big((t-s)\lambda\big)}{\lambda} \mathcal{F}_{D}F(v,w)(\lambda) \right\|_{L^2_{\lambda}(\mathbb{R}^+)}  \, ds \nonumber\\
    \lesssim & \int^t_0  (t-s) \left\|v(s)^{p-1}v(s)+\mathcal{O}(|v(s)|^{p-1}|w(s)|)+\mathcal{O}(|v(s)||w(s)|^{p-1}) \right\|_{L^2(\Omega)}  \, ds. 
    \end{align}
    
    We estimate one by one as follows.
    \begin{align*}
 \int^t_0  (t-s) \left\|v(s)^{p-1}v(s)\right\|_{L^2(\Omega)}\,ds
  \lesssim & \int^t_0  (t-s)\|v(s)\|_{L^6(\Omega)}^3\||x|^{\frac{4}{p+3}}v(s)\|^{p-3}_{L^\infty(\Omega)} \,ds \nonumber \\
    \lesssim & \int^t_0  (t-s)\|v(s)\|_{L^6(\Omega)}^3 \|v\|^{\frac{(p+1)(p-3)}{p+3}}_{L^{p+1}(\Omega)} \|v\|^{\frac{2(p-3)}{p+3}}_{\dot H^1_D(\Omega)}\,ds 
    \lesssim T^2 E_T^{\frac{3}{2}+\frac{2(p-3)}{p+3}}, 
         \end{align*}
 and
     \begin{align*}
    \int^t_0  (t-s) \left\||v(s)|^{p-1}|w(s)|\right\|_{L^2(\Omega)}\,ds 
    \lesssim &     \int^t_0 (t-s)\|v(s)\|_{L^{\frac{3(p-1)^2}{2(p-2)}}(\Omega)}^{p-1}\|w(s)\|_{L^{\frac{3}{1-s_c}}(\Omega)}\,ds\nonumber\\
    \lesssim &   \int^t_0 (t-s)\|v(s)\|^{\frac{4(p-2)}{p-1}}_{L^6(\Omega)} \|v(s)\|^{\frac{(p-3)^2}{p-1}}_{L^{\infty}(\Omega)}\|w(s)\|_{L^{\frac{3}{1-s_c}}(\Omega)}\,ds\nonumber \\
    \lesssim &     \int^t_0 (t-s)\|v(s)\|^{\frac{4(p-2)}{p-1}}_{L^6(\Omega)} 
    \|v\|^{\frac{(p+1)(p-3)^2}{(p-1)(p+3)}}_{L^{p+1}(\Omega)}\|v\|^{\frac{2(p-3)^2}{(p-1)(p+3)}}_{\dot H^1_D(\Omega)} \|w(s)\|_{L^{\frac{3}{1-s_c}}(\Omega)}\,ds\nonumber \\
    \lesssim &  T^{\frac{3}{2}}\, E_T^{\frac{2(p-2)}{p-1}+\frac{(p-3)^2}{(p-1)(p+3)} +\frac{(p-3)^2}{(p-1)(p+3)} }\|w\|_{L_t^2L_x^\frac{3}{1-s_c}([0,T)\times \Omega)},
        \end{align*}
and   
    \begin{align*}
    \int^t_0  (t-s) \left\||v(s)||w(s)|^{p-1}\right\|_{L^2(\Omega)}\,ds 
    \lesssim & \int^t_0  (t-s) \|w(s)\|_{L^{3(p-1)}(\Omega)}^{p-1}\|v(s)\|_{L^6(\Omega)}\,ds\nonumber\\
    \lesssim & \;  T\, E_T^{\frac12}\; \|w\|^{p-1}_{L_t^{p-1}L_x^{3(p-1)}([0,T)\times \Omega)}.  
    \end{align*}
 
  Therefore, we have 

     \begin{align}
 &  \left\|\int_0^t \frac{\sin((t-s)\sqrt{-\Delta_\Omega})}{\sqrt{-\Delta_\Omega}} F(v,w)(s)\,ds\right\|_{L^\infty_tL^2_x([0,T)\times \Omega)}  \nonumber\\    \lesssim & \; T^2 E_T^{\frac{3}{2}+\frac{2(p-3)}{p+3}  }
    + T^{\frac{3}{2}}\, E_T^{\frac{2(p-2)}{p-1}+\frac{2(p-3)^2}{(p-1)(p+3)} } \|w\|_{L_t^2L_x^\frac{3}{1-s_c}([0,T)\times \Omega)}   + T\, E_T^{\frac{1}{2}}   \|w\|^{p-1}_{L_t^{p-1}L_x^{3(p-1)}([0,T)\times \Omega)}  \nonumber \\ 
\lesssim & \;     T^{2+\frac{(7p-3)(1-s)}{2\left[4(2p-3)s-4(2p-3)+(p+3)(1-s_c)\right]} },
\end{align}
where $T$ determined by \eqref{est:T}. Therefore, using the interpolation between $L^2(\Omega)$ and $\dot{H}^1_{D}(\Omega)$, we have 
\begin{align*}
        &\left\|\int_0^t \frac{\sin((t-s)\sqrt{-\Delta_\Omega})}{\sqrt{-\Delta_\Omega}}F(v,w)(s)\,ds\right\|_{L^\infty_t\left([0,T);\dot{H}^s_{D}( \Omega)\right) } \\
        \lesssim& \left(T^{2+\frac{(7p-3)(1-s)}{2\left[4(2p-3)s-4(2p-3)+(p+3)(1-s_c)\right]} }\right)^{1-s}
        \left(    T^{\frac{(p+3)(1-s)}{2\left[4(2p-3)s-4(2p-3)+(p+3)(1-s_c)\right]} }\right)^s  \\
        \lesssim &  T^{2(1-s)+\frac{\left(7p-3-(6p-6)s\right)(1-s)}{2\left[4(2p-3)s-4(2p-3)+(p+3)(1-s_c)\right]} }. 
\end{align*}
Therefore we complete the proof of Theorem \ref{thm:GWP}. 

\def\cprime{$'$}

\end{document}